\documentclass{amsart}

\usepackage{amsmath, amsthm, latexsym, amssymb, amsfonts}
\usepackage{pgf,pgfarrows,pgfnodes,pgfautomata,pgfheaps}

\newtheorem{sen}{Statement}
\newtheorem{lemma}{Lemma}
\newtheorem* {Stheorem}{Svenonius theorem}
\newtheorem* {seq}{Corollary}

\begin{document}
\title{The Lattice of Definability for Integers with Successor}
\author[A.L. Semenov, S.F. Soprunov]{A.L. Semenov, S.F. Soprunov}
\begin{abstract} 
In the present paper we describe the lattice of definability for integers with successor (relation $y = x +1$). The lattice elements (reducts) constitute three (naturally described) infinite series. Our proof uses Svenonius theorem.
\end{abstract}

\maketitle

Let  $S$ be any set of relations on a universe $A$ and $R$ is a name of a relation on $A$. To define a relation $R$ through $S$ in a logical language $L$ means:

(1) to give names to some relations from $S$ and 

(2) to write a formula in the language $L$ using the given names as extra-logical symbols that is equivalent to $R$ (on $A$).

In this paper $L$ will be the first-order logic with equality, we consider countable universe, and countable or finite sets of relations.

The \emph{(definability) closure} is the operation of extension of a set of relations $S$ with all relations definable through $S$. This operation is a closure operation in the usual topological and algebraic sense. Closed sets of relations we call \emph{definability spaces}, the set $S$ is a \emph{base} of the definability closure of $S$.

A permutation $\varphi$ of $A$ \emph{preserves} a relation $R$ iff $R(\overline a) \equiv R(\varphi (\overline a))$ where $\overline a$ is a tuple of elements from $A$ and $\varphi (\overline a)$ is the tuple of their images under permutation $\varphi$. A permutation preserves a set of relations $S$ if it preserves all relations from $S$, a collection of permutations $F$ preserves $S$ if any permutation from $F$ preserves $S$.

With any set of relations $S$ we can associate the group $G_S \subseteq Sym(A)$ of permutations of the set $A$ preserving $S$. It's obvious, that $S_1 \subseteq S_2 \Rightarrow G_{S_1} \supseteq G_{S_2}$ but usually we can't recover a definability space from the corresponding subgroup of $Sym(A)$.

Let $S_1$ be a definability space with the universe $A$, $B \subseteq A$ and let $S_2$ be the set of restrictions on $B$ of all relations from $S_1$. Let us give names to relations from a finite subset $F$ of $S_1$ and use the same names for the restrictions of relations from $F$ on $B$. Let us take any formula (using the chosen names). It defines a relation on $A$, and on $B$ as well. The second relation can be not the restriction of the first on B. But if it is so for any formula (and any names), then $S_2$ is called \emph{elementary restriction} of $S_1$ (and $S_1$ is \emph{elementary extension} of $S_2$).
It's obvious that an elementary restriction of a definability space is a definability space.


If $S_1$ and $S_2$ are sets of relations on the same universe, then $S_1 \succcurlyeq S_2$ denotes that the definability space, generated by $S_2$ is the subset of the definability space, generated by $S_1$. If $S_1 \succcurlyeq S_2$ and $S_2 \succcurlyeq S_1$ then we write $S_1 \approx S_2$.

We call a definability space \emph{countable} if it is countable or finite and its universe is countable. 

Our main tool will be the Svenonius theorem \cite{sv}, which in our case may be formulated as following:

\begin{Stheorem}\label{Sven}
Let $S, S^+$ are countable definability spaces on a universe $A$ , and $S \subset S^+$. Then the following are equivalent for any relation $R \in S^+$. 

(a) $R \in S$.
 
(b) For any ${S^+}'$ -- countable elementary extension of $S^+$ the following holds. Assume that $S' \subset {S^+}', R' \in {S^+}'$ and the restriction of $S'$ on $A$ is $S$, the restriction $R'$ on $A$ is $R$. Then the group of permutations on the universe of ${S^+}'$, preserving all relations from $S'$, preserves $R'$. 
\end{Stheorem}

Hereby the Svenonius theorem states that if we consider not only subgroups of permutations of an original space but also permutations of its elementary extensions then the definability space can be recovered.

Permutation groups were used for description of definability lattices in many cases (see e.g. \cite{MAC} ). The first remarkable result – the description of definability lattice for rationals with order was obtained in \cite{FRA}. It was rediscovered many times in (see \cite{cam}, \cite{sem_much}).

The main goal is to describe the definability lattice for integers with successor.

These examples were discussed in \cite{csr}.

 \section{Integer numbers $Z$ with the successor relation $'$}

In the present section we consider the structure $<\!\!Z, \{'\}\!\!>$ -- integer numbers with the successor relation. We will show that the lattice of the definability space of $<\!\!Z, \{'\}\!\!>$ is rather simple. For any natural number $n$ we denote the relations $|x_1-x_2|=n$ by $A_{0,n}$, the relation $x_1-x_2=x_3-x_4=n \lor x_1-x_2=x_3-x_4=-n$ by $A_{1,n}$, and the relation $x_1-x_2=n$ by $A_{2,n}$. We set $A_{i,j} \land A_{k,l} = A_{m,n}$ where $m=max\{k,m\}, n=$ greatest common divisor of $j,l$; and $A_{i,j} \lor A_{k,l} = A_{m,n}$ where $m=min\{k,m\}, n=$ least common multiple of $j,l$. We will demonstrate that if relation $R$ is definable in $<\!\!Z, \{'\}\!\!>$  and is not identically true (false) then $R \approx A_{i,n}$ for some natural $n$ and $i \leqslant 2$. 

Unlike the previous section there exist different countable extensions of the original structure, but there is the universal countable elementary extension of $<\!\!Z, \{'\}\!\!>$ -- the structure $M_Z$ defined as follows: the domain of $M_Z$ is $Z \times S$ where $S$ is some countable set and the relation $'$ is specified as $(x,y)=(x_1, y_1)' \iff x=x_1+1, y=y_1$.

It's easy to see that any countable elementary extension of $M_Z$ is isomorphic to $M_Z$ so, according to Svenonius theorem we can limit our consideration to permutations on $M_Z$. For any $a \in M_Z$ we denote by $a^1 (a^2)$ the first (second) component of $a$. Two members $a,b \in M_Z$ are called to be in the same  \emph{galaxy}, if $a^2=b^2$. For any $a \in M_Z, z \in Z$ by $a \pm z$ we denote the item $(a^1 \pm z, a^2)$. We will also need the ordered set $Z_\infty = Z \cup \{\infty\}$, the order on $Z$ is natural and $z<\infty$ for any $z \in Z$. We define the function of absolute value ($||$) on $Z_\infty$: it is natural on $Z$ and $|\infty|=\infty$. The subtraction function  ($-$) maps $M_Z \times M_Z$ on $Z_\infty$ as follows: $a-b=a^1-b^1$ if $a^2=b^2$ and equal  to $\infty$ if $a^2 \ne b^2$. The expression $a>b$ for $a,b \in M_Z$ is simply an abbreviation for $\infty >a-b>0$. If $m$ is natural number then we call two vectors $\overline a, \overline b \in M_Z$ of the same length \emph{$m$-indistinguishable}, if $|a_i-a_j|<m$ or $|b_i-b_j|<m$ implies $a_i - a_j=b_i-b_j$ for any $i,j$. 

In the present section a \emph{permutation} is a permutation on the domain of $M_Z$. A permutation $f$ is called  \emph{shift} if $f(a) - f(b) = a-b$  for any $a,b \in M_Z$. It is clear that the set of all permutations preserving $'$ is the set of shifts. By $\Gamma$ we denote the group of all shifts.

\begin{lemma} \label{lem50}
Suppose that $a_0, \dots, a_{n-1}, b_0, \dots, b_{n-1} \in M_Z$ are such, that for all $i,j$ holds $a_i-a_j=b_i-b_j$. Then the partial mapping $f(a_i)=b_i$, for all $i$ can be extended to a shift
\end {lemma}
\begin{lemma}\label{lem60}
For any formula $R$ in the signature $\{'\}$ there exists such a natural number $w$  that for any two $w$-indistinguishable tuples $\overline a, \overline b \in M_Z$ holds $R(\overline a)=R(\overline b)$.
\end{lemma}
Lemma \ref{lem50}  is very simple, to proof it you can consider elements from the same or different galaxies. Lemma \ref{lem60}  is simple too, but we are giving a proof of it.
\begin{proof}
Let a formula $Q(w, \overline x, \overline y)$ in the signature $\{+, <\}$ express that the tuples $\overline x, \overline y$ are $w$-indistinguishable. Then the statement of lemma \ref{lem60} can be expressed as $(\exists w)(\forall \overline x)(\forall \overline y)(Q(w, \overline x, \overline y) \Rightarrow (R(\overline x) \equiv R(\overline y)))$. Consider a countable non-standard extension $M_0$ of the structure $<\!\!Z, \{+, <\}\!\!>$ and a non-standard number $w_0$ in $M_0$. For any tuples $\overline a, \overline b$ from the sentence $Q(w_0, \overline a, \overline b)$ follows that for any standard $k$ holds $a_i - a_j=k$ iff $b_i - b_j=k$. So, according to the lemma \ref{lem50}, the mapping $f(a_i)=b_i$ can be extended to a shift and so the sentence $(\forall \overline a)(\forall \overline b)(Q(w_0, \overline a, \overline b) \Rightarrow (R(\overline a) \equiv R(\overline b)))$ is true in $M_0$. Then the sentence $(\forall \overline a)(\forall \overline b)(Q(m, \overline a, \overline b) \Rightarrow (R(\overline a) \equiv R(\overline b)))$ is true for a standard number $m$ as well. Because the structures $M_0$ and $M_Z$ are isomorphic (as structures with the only relation $'$) , then the statement of lemma \ref{lem60} holds in $M_Z$.

\end{proof}

Let a group of permutations $\Gamma'$ include the group of shifts $\Gamma$. Two members $z_1,z_2 \in Z_\infty$ are called \emph{equivalent} (respectively to $\Gamma'$) if for some $\gamma \in \Gamma', a,b \in M_Z$, holds $a-b=z_1, \gamma(a) - \gamma(b)=z_2$. The equivalence class (respectively to $\Gamma'$) of $z$ we denote by $K_z$. A number $z \in Z$ is called \emph{regular} (respectively to $\Gamma'$) if $K_z$ is finite and $\infty \not \in K_z$.

For example if $\Gamma'=\Gamma$ then the equivalence is trivial, if $\Gamma'$ is generated by $\Gamma$ and $f(x)=-x$ then any $z \in Z$ is regular and $K_z=\{z, -z\}$.

\begin{lemma}\label{lem70}
\mbox{ }

(i) If $z_1$ and $z_2$ are regular numbers, then so is $z_1 \pm z_2$.

(ii)  Greatest common divisor of two regular numbers is a regular number.
\end{lemma}
\begin{proof} (i) For any $f \in \Gamma'$ holds $f(a+z_1+z_2)-f(a)=f(a+z_1+z_2)-f(a+z_1)+f(a+z_1)-f(a); f(a+z_1+z_2)-f(a+z_1) \in K_{z_2}, f(a+z_1)-f(a) \in K_{z_1}$, so $K_{z_1+z_2}$ is finite and does not contain $\infty$. 

(ii)  follows from (i).
\end{proof}
\begin{lemma}\label{lem80} 
Let a group of permutations $\Gamma'$ include $\Gamma$ and $d$ is  the greatest common divisor of all numbers regular respectively to $\Gamma'$. Then $K_d=\{d\}$ or $K_d=\{d, -d\}$. 

Moreover, if $K_{d} =\{d\}$ then $K_z=\{z\}$ for any $z$ which is a multiple of $d$, if $K_{d}=\{d, -d\}$ then $K_z=\{z, -z\}$ for any $z$ which is a multiple of $d$.
\end{lemma}
\begin{proof} 
Denote by $D$ a number equivalent to $d$ with the maximum absolute value. Then $D=N \cdot d$ or $D=-N \cdot d$ for a natural number $N$. Suppose that $N>1$. Choose $\gamma \in \Gamma', a,b \in M_Z$ such as $a-b=D, \gamma(a) - \gamma(b)=d$. For any $0 \leqslant k < N$ denote $C_k=\{c_{k,i} \mid  c_{k,i}=a+k \cdot d + i\cdot D, i \in Z\}$. The collection $\{C_k\}$ is the partition of the set $\{a+z \cdot d \mid z \in Z\}$. Since $d$ is regular, the expression $\gamma (a) - \gamma (c)$ is finite and it is a multiple of $d$ for any $c \in C_k, 0 \leqslant k < N$. So the collection $\{\gamma(C_k)\}$ is the partition of the set $\{\gamma(a)+z \cdot d \mid z \in Z\}$.

Consider the set $E=\{\gamma(a), \gamma(a)+d, \dots, \gamma(a)+(N-1) \cdot d\}$.  Since   $\gamma(a),\gamma(b) \in \gamma(C_0) \cap E$ and there are only $N$ elements in $E$, then there is $0 \leqslant k' < N$ such that $\gamma(C_{k'}) \cap E = \varnothing$. Since the absolute value of $D$ is maximal in the equivalence class of $d$ then $\gamma(C_{k'})$ has to lie on one side from the segment $E$: either $\gamma(c) < \gamma(a)$ for any $c \in C_{k'}$, or $\gamma(a)+(N-1) \cdot d < \gamma (c)$ for any $c \in C_{k'}$. Otherwise there is such $c_{k',i}$ that $|\gamma(c_{k',i+1}) - \gamma(c_{k',i})| > D$, but $c_{k',i+1} - c_{k',i} = D$ and $D$ has maximal absolute value in its class of equivalence. 

Suppose that $\gamma(c) < \gamma(a)$ for any $c \in C_{k'}$ (another case is similar). There is $0 \leqslant k'' < N$ such that the set $\{c \in C_{k''} \mid \gamma(c) > \gamma(a)\}$ is infinite. Then the value of $|\gamma (a+k' \cdot d +z \cdot D) - \gamma (a+k'' \cdot d +z \cdot D)|$ can be arbitrary big  when  $z \in Z$, in contradiction with regularity of  $(k' - k'') \cdot d$. So $N=1$ and $K_d =\{d\}$ or $K_d =\{d, -d\}$.

If $K_d =\{d\}$ then it is clear that $K_z=\{z\}$ for any $z$ multiple of $d$. Suppose that $K_d=\{d,-d\}$ and $z=n\cdot d$ where $n$ is a natural number (case of $z=-n\cdot d$ is absolutely similar). For any $0 \leqslant i <n$ and any $\gamma \in \Gamma', a \in M_Z$ holds $\gamma(a+(i+1)\cdot d) - \gamma(a+i \cdot d)=d$ or $\gamma(a+(i+1)\cdot d) - \gamma(a+i \cdot d)=-d$. Moreover, since $\gamma(a+(i+2)\cdot d) \ne \gamma(a+i \cdot d)$ all the differences have the same sign, i.e. $\gamma(a+n\cdot d) - \gamma(a)=n \cdot d$ or $\gamma(a+n\cdot d) - \gamma(a)=-n \cdot d$ for any $a \in M_Z$, i.e. $K_z \subset \{z, -z\}$. Since there are $\gamma \in \Gamma', a \in M_Z$, such as $\gamma(a+d) - \gamma(a)=-d$, then $\gamma(a+z) - \gamma(a) = -z$ and $K_z=\{z, -z\}$.
\end{proof}
Hereby if a group of permutations $\Gamma'$ includes $\Gamma$ and $d$ is  the greatest common divisor of all regular respectively to $\Gamma'$ numbers and $f \in \Gamma'$ then there are three essential possibilities (the case when there is no regular number is trivial, as we will show later): (1) $f(a+n\cdot d)-f(a)=n\cdot d$ for any $a \in M_Z$ and any natural number $n$ (such permutations are called permutations of  \emph{first type}), (2) $f(a+n\cdot d)-f(a)=-n\cdot d$ for any $a \in M_Z$ and  any natural number $n$ (such permutations are called permutations of  \emph{second type}), and (3) for any $a \in M_Z$ and  any natural number $n$ holds $f(a+n\cdot d)-f(a)=n\cdot d$ or $f(a+n\cdot d)-f(a)=-n\cdot d$, each of this equalities is realized by some $a, n$ (such permutations are called permutations of  \emph{third type}).

If $K_d=\{d\}$ then any permutation $f \in \Gamma'$ belongs to the first type. If $K_d=\{d,-d\}$, then a permutation $f \in \Gamma'$ may belong to the first, second, or third type.

The point of the following lemma is that if some differences between items of a vector $\overline a$ are non-regular (respectively to a group, preserving a relation $R$) then they can be replaced by infinity without changing the value of $R(\overline a)$.

\begin{lemma}\label{lem90}
Let $\Gamma' \supset \Gamma$, an $n$-ary relation $R$ be definable in $<\!\!Z, \{'\}\!\!>$, $\Gamma'$ preserve $R$ and $\overline a=(a_0,\dots,a_{n-1}) \in M_Z$. Then there is such a vector $\overline b=(b_0,\dots,b_{n-1}) \in M_Z$
that

(i) $R(\overline a) \equiv R(\overline b)$.

(ii) if the difference $a_i-a_j$ is not regular respectively to $\Gamma'$ then $b_i-b_j= \infty$.

(iii) if the difference $a_i-a_j$ is regular then $|a_i-a_j|=|b_i-b_j|$. Moreover if $\Gamma'$ contains permutations of the first type only, then $a_i-a_j=b_i-b_j$; if $\Gamma'$ doesn't contain a permutation of the third type, then either $a_i-a_j=b_i-b_j$ for any $i,j$ with regular difference $a_i-a_j$ or $a_i-a_j=b_j-b_i$ for any $i,j$ with regular difference $a_i-a_j$.

\end{lemma}
\begin{proof}
We prove by induction on number of such pairs $i,j$ that $a_i-a_j$ is finite and not regular. Suppose that $a_0-a_1$ finite and not regular. We'll construct such vector $\overline b$ that

(a)  $R(\overline a) \equiv R(\overline b)$;
 
(b) $b_0-b_1=\infty$;

(c) for any $i,j < n$ if $a_i-a_j=\infty$ then $b_i-b_j=\infty$;

(d) for any $i,j < n$ if $b_i-b_j=\infty$ then $a_i-a_j$ is not regular;

(e) for some $\gamma \in \Gamma'$ and any $i,j < n$ if $b_i-b_j<\infty$ then $b_i-b_j=\gamma(a_i)-\gamma(a_j)$;

Let $w$ be the width of relation $R$ and $w'$ is the maximal absolute value of regular differences $|a_i - a_j|, (i,j < n)$. We claim, that there is such permutation $\gamma \in \Gamma'$ that $|\gamma(a_0)-\gamma(a_1)| > n \cdot \max(w,w')$ and $|\gamma(a_i) - \gamma(a_j)|>w$ if $a_i-a_j=\infty$. 

In fact, the difference $a_0-a_1$ is not regular, so $|f(a_0)-f(a_1)| > n \cdot \max(w,w')$ for some $f \in \Gamma'$. Let us choose such a shift $s$, that (1) $s(a_0)=a_0, s(a_1)=a_1$ (2) if $a_i - a_j = \infty$ then $|f(s(a_i)) - f(s(a_j))| > w$. So we can take $f \circ s$ as $\gamma$. 

Due to lemma \ref{lem60} we can find such a vector $\overline a'$ for the vector $\gamma(\overline a)$ that (1) $R(\overline a) \equiv R(\overline a')$ (2) if $a_i - a_j < \infty$ then $a'_i - a'_j=\gamma(a_i) - \gamma(a_j)$ (3) if $a_i - a_j = \infty$ then $a'_i - a'_j = \infty$.

If $a'_0-a'_1=\infty$ then we can choose the vector $\overline a'$ as the vector $\overline b$: it's easy to see that conditions (a)--(e) hold.

If $a'_0 - a'_1 < \infty$ then $a'_0, a'_1$ belong to the same galaxy $U$. Suppose that $a'_0 < a'_1$. Let $c_0<\dots<c_k$ be all items of the vector $\overline a'$ from the galaxy $U$. There is such $c_m$, that $a'_0 \leqslant c_m<c_{m+1} \leqslant a'_1$ and $c_{m+1}- c_{m}>\max(w,w')$. Find a vector $\overline b$ such that (1) $b_i=a'_i$ if $a'_i \not \in U$ or $a'_i \leqslant c_m$ (2) all items $\{b_i \mid a_i > c_m, a_i \in U\}$ lie in a new galaxy which doesn't contain items from the vector $\overline a'$ and (3) $b_i-b_j=a'_i-a'_j$ if $a'_i, a'_j > c_m, a_i, a_j \in U$. Due to lemma \ref{lem60} holds  $R(\overline a')\equiv R(\overline b)$.
Because  $c_{m+1}- c_{m}>w'$ then $b_i-b_j$ is regular iff $a_i-a_j$ is regular, so it's easy to see that conditions (a) -- (e) hold.  
\end{proof}

From now by $\Gamma_R$ we denote the group of permutations preserving the relation $R$.

\begin{seq}
If  a relation $R$ is definable in $<\!\!Z, \{'\}\!\!>$ and no number is regular respectively to $\Gamma_R$ then $R$ is constantly true (false).
\end{seq}
Recall that for any natural number $n$ by $A_{2,n}$ we denote the relation $x_1-x_2=n$, by $A_{1,n}$ we denote the relation $x_1-x_2=x_3-x_4=n \lor x_1-x_2=x_3-x_4=-n$, and by $A_{0,n}$ we denote the relation $|x_1-x_2|=n$.

\begin{sen} Suppose that relation $R$ is definable in $<\!\!Z, \{'\}\!\!>$, $d$ is the common greatest divisor of regular respectively to $\Gamma_R$ numbers. Then 

(i) if $\Gamma_R$ doesn't contain a permutation of second or third types then $R \approx A_{2,d}$.

(ii) if $\Gamma_R$ doesn't contain a permutation of third type but contains a permutation of second type then $R \approx A_{1,d}$.

(iii) if $\Gamma_R$ contains a permutation of third type then $R \approx A_{0,d}$.
\end{sen}
\begin{proof} It's clear that $R \succcurlyeq A_{2,d}$ ($R \succcurlyeq A_{1,d},  R  \succcurlyeq A_{0,d}$ respectively): it's easy to note that if (i) holds then any permutation from $\Gamma_R$ preserves $A_{2,d}$, if (ii) holds then it preserves $A_{1,d}$, and  if (iii) holds then it preserves $A_{0,d}$. 

To prove the reverse sentence we need to show that any permutation preserving  $A_{2,d} (A_{1,d}, A_{0,d})$ belongs, if the corresponding condition holds,  to $\Gamma_R$.

Denote by $\Gamma'$ the set of such permutations $f$, that $|f(a+d)-f(a)|=d$ for any $a \in M_Z$. It's clear that $\Gamma'$ is the set of all permutations preserving $A_{0,d}$; the subgroup of $\Gamma'$ containing permutations of the first and second type is the set of all permutations preserving $A_{1,d}$; the subgroup of $\Gamma'$ containing permutations of the first type is the set of all permutations preserving $A_{2,d}$.

Proof (i). Suppose that there is a permutation $f$ of the first type in $\Gamma' \setminus \Gamma_R$. Then there is a vector $\overline a=(a_0,\dots,a_{n-1}) \in M_Z$ such that $R(\overline a) \not \equiv  R(f(\overline a))$. By definition of $\Gamma'$, if $a_i-a_j$ is non-regular  respectively $\Gamma_R$, then the difference $f(a_i)-f(a_j)$ is non-regular as well; by definition of the first type if $a_i-a_j$ is regular respectively $\Gamma_R$, then $f(a_i)-f(a_j)= a_i-a_j$. We use lemma \ref{lem90}  to choose vectors $\overline b, \overline c$ corresponding to vectors $\overline a$ and $f(\overline a)$ respectively. Because $\Gamma_R$ contains permutations of the first type only,  regular differences in vectors $\overline b, \overline c$ are the same. So, $R(\overline b) \equiv R(\overline c)$ according lemma \ref{lem60}, contradiction. 

Proof (ii). Suppose that there is a permutation $f$ of the first or second type in $\Gamma' \setminus \Gamma_R$. We choose vectors $\overline a, f(\overline a), \overline b, \overline c$ as in the case (i). Permutation $f$ is the permutation of the first or second type, so if $b_i-b_j=c_i-c_j$ for some regular difference $b_i-b_j$ then the same equality holds for any regular difference. This contradict the lemma \ref{lem60}.  

If $b_i-b_j=c_j-c_i$ for some regular difference $b_i-b_j$ then the same equality holds for any regular difference. There is a permutation $\gamma$ of second type in $\Gamma_R$, i.e. $\gamma (t+z \cdot d) - \gamma(t)=-z \cdot d$ for any $z \in Z, t \in M_Z$. Fix some $t \in M_z$ and consider the set $S=\{t+d \cdot z \mid z \in Z\}$. 
Choose a vector $\overline s$ in $S$ such that (1)if $|c_i - c_j| < \infty$ then $s_i-s_j=c_i-c_j$, (2) if $|c_i - c_j| = \infty$ then $|s_i - s_j| > w$ where $w$ is the width of the relation $R$. By lemma \ref{lem60} holds $R(\overline c) \equiv R(\overline s)$. Because $s_i-s_j=\gamma(s_j)-\gamma(s_i)$ for any $i,j < n$, vectors  $\overline b$ and $\gamma (\overline s)$ are  $w$-indistinguishable, which contradicts to lemma \ref{lem60}.

Proof (iii). Suppose that there is a permutation $f$ in $\Gamma' \setminus \Gamma_R$. We choose vectors $\overline a, f(\overline a), \overline b, \overline c$ as in the case (i). If the difference $b_i-b_j$ is regular then $b_i-b_j=c_i-c_j$ or $b_i-b_j=c_j-c_i$. All differences between items of $\overline b$ from same galaxy are regular, so if $b_i-b_j=c_i-c_j$ and $b_k - b_l=c_k-c_l$ then $b_i$ and $b_k$ belongs to different galaxies. There is a permutation $\gamma$ of third type in $\Gamma_R$, i.e. for some $t_1,t_2 \in M_Z$ and any $z \in Z$holds $\gamma (t_1+z \cdot d) - \gamma(t_1)=z \cdot d, \gamma (t_2+z \cdot d) - \gamma(t_2)=-z \cdot d$. 

Consider sets $S_1=\{t_1+z \cdot d \mid z \in Z\}, S_2=\{t_2+z \cdot d \mid z \in Z\}$. Choose such collection $\{s_0,\dots,s_{n-1}\} \subset S_1 \cup S_2$ that (1) if $|c_i - c_j| < \infty$ then  $s_i-s_j=c_i - c_j$; (2) if $|c_i - c_j|=\infty$ then $|s_i - s_j| >w$ and $|\gamma(s_i) - \gamma(s_j)| >w$, where $w$ is the width of the relation $R$; (3) if the difference $b_i - b_j$ is a regular number and $b_i - b_j=c_i-c_j$ then $s_i, s_j \in S_1$; (4) if the difference $b_i - b_j$ is a regular number and $b_i - b_j=c_j-c_i$ then $s_i, s_j \in S_2$. By lemma \ref{lem60} holds $R(\overline s) \equiv R(\overline c)$. Because for any regular difference $b_i-b_j$ holds $\gamma(s_i)-\gamma(s_j)=b_i-b_j$, vectors $\overline b$ and $\gamma (\overline s)$ are  $w$-indistinguishable, which contadict to lemma \ref{lem60}.
\end{proof}

\end{document}